\documentclass[12pt,A4paper]{amsart}

\usepackage{amsthm,amssymb,amsmath,amscd,fullpage,enumerate,hyperref}
\usepackage[all]{xy}
\usepackage[dvips]{graphicx}

\swapnumbers
\numberwithin{equation}{section}

\newenvironment{pf}[1][Proof.]{\noindent \textbf{#1:} }{}
\newenvironment{enui}{\begin{enumerate}[(i)]}{\end{enumerate}}

\newtheorem {lemma}[equation]       {Lemma}
\newtheorem{prop}[equation]{Proposition}
\newtheorem{thm}[equation]{Theorem}

\newtheorem* {Theorem*}     {Theorem}
\newtheorem{cor}[equation]{Corollary}
\numberwithin{equation}{section}

\newtheorem{claim}[equation]{Claim}

\newtheorem*{Lemma*}{Lemma}

\theoremstyle{definition}

\theoremstyle{remark}

\newtheorem{rmk}[equation]{Remark}
\newtheorem{rmks}[equation]{Remarks}

\newtheorem* {Remark*}      {Remark}
\newtheorem*{Rmk}{Remark}
\newtheorem*{Rmks}{Remarks}
\newtheorem {Annoying Remark}[equation]      {Annoying Remark}
\newtheorem* {Example*}          {Example}

\newcommand\Lie{\operatorname{Lie}}

\def \g {{\mathfrak g}}
\def \h {{\mathfrak h}}

\def \R {{\mathbb R}}

\DeclareMathOperator \Ad {Ad}

\def\x{\times}

\def\om{\omega}

\def\sub{\subseteq}

\def\pr{\operatorname{pr}}

\newcommand\wt[1]{{\widetilde{#1}}}
\newcommand\iso{\cong}

\newcommand\reff[1]{(\ref{#1})}
\newcommand\im{{\operatorname{im}}}
\newcommand\nn{{\nonumber}}
\newcommand\then{\Rightarrow}
\newcommand\follows{\Leftarrow}
\newcommand\Wlog{w.l.o.g.~}

\newcommand{\BAR}[1]{{\overline{#1}}}
\renewcommand\phi{\varphi}
\newcommand\wrt{w.r.t.~}
\newcommand\si{\sigma}
\newcommand\SympRep[1]{\operatorname{SympRep}_{\leq#1}}
\newcommand\SympRepProp[1]{\operatorname{SympRep}^{\operatorname{prop}}_{\leq#1}}
\newcommand\symprep[1]{\operatorname{SympRep}_{#1}}
\newcommand\SYMPREP[1]{\wt{\operatorname{SympRep}}_{#1}}
\newcommand\symprepprop[1]{\operatorname{SympRep}^{\operatorname{prop}}_{#1}}
\newcommand\HamEx[1]{\operatorname{Ham}^{\operatorname{ex}}_{#1}}
\newcommand\HamExProp[1]{\operatorname{Ham}^{\operatorname{ex,prop}}_{#1}}
\newcommand\HamCrit[1]{\operatorname{Ham}^{\operatorname{crit}}_{#1}}
\newcommand\HamContr[1]{\operatorname{Ham}^{\operatorname{contr}}_{#1}}
\newcommand\HamContrProp[1]{\operatorname{Ham}^{\operatorname{contr,prop}}_{#1}}
\newcommand\Model[1]{\operatorname{Model}_{#1}}

\newcommand\I[1]{\iota^{#1}}
\newcommand\IProp[1]{\iota^{#1,\operatorname{prop}}}
\newcommand\II[2]{\operatorname{I}_{#1}^{#2}}
\newcommand\III[1]{i_{#1}}
\newcommand\IIII[1]{j_{#1}}
\newcommand{\hhat}[1]{\widehat{#1}}
\newcommand\muL{\mu^L}

\newcommand\muD{\mu^D}

\newcommand\psiD{\psi^D}
\newcommand\psiL{\psi^L}
\newcommand\cont{\supseteq}
\newcommand\Stab[2]{\operatorname{Stab}^{#1}_{#2}}
\newcommand\isotr[2]{\rho^{#1,#2}}
\newcommand\quot[2]{\BAR\rho^{#1,#2}}

\newcommand\F[1]{\mathcal{F}_{#1}}
\newcommand\Tstar[1]{T^*_{#1}}
\newcommand\Trans[1]{\operatorname{Act}^{\operatorname{trans}}_{#1}}
\newcommand\Sub[1]{\operatorname{Sub}^{\operatorname{cl}}_{#1}}

\begin{document}

\title{Classification of momentum proper exact Hamiltonian group actions and the equivariant Eliashberg cotangent bundle conjecture}

\author{Fabian Ziltener}
\address{ETH Zurich, Department of Mathematics, R\"amistrasse 101, 8092 Zurich, Switzerland}
\email{fabian.ziltener@math.ethz.ch}

\date{\today}

\begin{abstract}
Let $G$ be a compact and connected Lie group. The Hamiltonian $G$-model functor maps the category of symplectic representations of closed subgroups of $G$ to the category of exact Hamiltonian $G$-actions. Based on previous joint work with Y.~Karshon, the restriction of this functor to the momentum proper subcategory on either side induces a bijection between the sets of isomorphism classes. This classifies all momentum proper exact Hamiltonian $G$-actions (of arbitrary complexity). 

As an extreme case, we obtain a version of the Eliashberg cotangent bundle conjecture for transitive smooth actions. As another extreme case, the momentum proper Hamiltonian $G$-actions on contractible manifolds are exactly the symplectic $G$-representations, up to isomorphism.
\end{abstract}

\maketitle
\tableofcontents

\section{The main result and applications}

Let $G$ be a compact and connected Lie group. We call a Hamiltonian $G$-action \emph{momentum proper} iff every momentum map for the action is proper. The purpose of this article is to classify the momentum proper exact Hamiltonian $G$-actions in terms of the momentum proper symplectic representations of closed subgroups of $G$. To this end I provide a bijection between the sets of equivalence classes of such representations and of such Hamiltonian actions. (See Corollary \ref{cor:class} below.)

The bijection is induced by the $G$-model functor. This is a functor between the category of symplectic representations of closed subgroups of $G$ and the category of exact Hamiltonian $G$-actions.

In order to define these categories and the Hamiltonian $G$-model functor, we need the following. For every $g\in G$ we denote by
\[c_g:G\to G,\quad c_g(a):=gag^{-1},\]\label{c g}
the conjugation by $g$. We define $\SympRep{G}$\label{SympRep} to be the following category:
\begin{itemize}
\item Its objects are the tuples $(H,\rho)=\big(H,V,\si,\rho\big)$, where $H$\label{H} is a closed subgroup of $G$, $(V,\si)$ \footnote{In the article \cite{KZ} we used the notation $\om_V$ for the symplectic form $\si$ on $V$. For simplicity I am using $\si$ here. To help the reader navigate through this article, I have included a list of symbols at the end of this article.} is a (finite dimensional) symplectic vector space and $\rho$\label{rho} is a symplectic $H$-representation on $V$ \wrt $\si$. 
\item Its morphisms between two objects $(H,\rho)$ and $(H',\rho')$ are pairs $(g,T)$\label{g T}, where $g\in G$ and $T:V\to V'$ is a linear symplectic map, 
such that
\begin{eqnarray}\label{eq:gHg -1}&c_g(H)=H',&\\
\label{eq:T rho h}&T\rho_h=\rho'_{c_g(h)}T,\quad\forall h\in H.&
\end{eqnarray}
(The dimension of $V'$ may be bigger than the dimension of $V$. In this case $T$ is not surjective.) The composition of two morphisms is defined by
\begin{equation}\label{eq:g' T' g T}(g',T')\circ(g,T):=\big(g'g,T'T\big).\end{equation}
\end{itemize}

\begin{Rmk} A morphism $(g,T)$ is an isomorphism in the sense of category theory if and only if $T$ is surjective (and hence bijective). In this case the inverse of $(g,T)$ is given by $(g^{-1},T^{-1})$.
\end{Rmk}

Let $\psi=\big(M,\om,\psi\big)$\label{M om psi} be a Hamiltonian $G$-action\footnote{By this we mean that $(M,\om)$ is a symplectic manifold and $\psi$ is a Hamiltonian $G$-action on $M$.}. We call $\psi$ exact iff there exists a $\psi$-invariant primitive of $\om$.\footnote{By this we mean a $\psi$-invariant one-form whose exterior derivative equals $\om$. The action $\psi$ is exact if $\om$ is exact, because we assume that $G$ is compact. (We obtain a $\psi$-invariant primitive from an arbitrary primitive by averaging \wrt the Haar measure on $G$.)} We call $\psi$ \emph{momentum proper} iff every momentum map\footnote{By definition, every momentum map $\mu$ for a Hamiltonian $G$-action $\psi$ on a symplectic manifold $(M,\om)$ is equivariant \wrt $\psi$ and the coadjoint $G$-action $\Ad^*$ on $\Lie(G)^*$. This means that $\mu(\psi_g(x))=\Ad^*(g)(\mu(x))$, for every $g\in G$ and $x\in M$, where $\psi_g:=\psi(g,\cdot)$.} for $\psi$ is proper.

We define $\HamEx{G}$\label{Ham Ex G} to be the following category:
\begin{itemize}
\item Its objects are the exact Hamiltonian $G$-actions $\big(M,\om,\psi\big)$ with $M$ connected (and without boundary\footnote{In this article every manifold is assumed to have empty boundary.}).
\item Its morphisms between two objects $(M,\om,\psi)$ and $(M',\om',\psi')$ are proper symplectic embeddings $\Phi$ from $M$ to $M'$ that intertwine $\psi$ with $\psi'$.\footnote{This means that $\psi'_g\circ\Phi=\Phi\circ\psi_g$ for every $g\in G$.} (The dimension of $M'$ may be bigger than the dimension of $M$.) Composition is the composition of maps.
\end{itemize}
\begin{Rmk} The isomorphisms between two objects are equivariant symplectomorphisms.
\end{Rmk}
We define the \emph{Hamiltonian $G$-model functor}\label{Model G}
\[\Model{G}:\SympRep{G}\to\HamEx{G}\]
as follows:
\begin{itemize}
\item For every object $(H,\rho)$ of $\SympRep{G}$ we define\label{Y rho om rho psi rho}
\[\Model{G}(H,\rho)=\big(Y_\rho,\om_\rho,\psi_\rho)\]
to be the centred Hamiltonian $G$-model action induced by $(H,\rho)$.\footnote{In the article \cite{KZ} we used the notation $(Y,\om_Y)$ instead of $(Y_\rho,\om_\rho)$. To make the dependence on $\rho$ explicit, I am using $(Y_\rho,\om_\rho)$ here. To help the reader navigate through this article, I have included a list of symbols at the end of this article.} This action is defined as follows. (For details see \cite[Section 3]{KZ}.) We define $\psiD_\rho$\label{psi D rho} to be the diagonal $H$-action on $T^*G\x V$ induced by the right translation on $G$ and by $\rho$. We denote by $\g,\h$\label{g h} the Lie algebras of $G,H$ and by
\begin{equation}\label{eq:nu rho}\nu_\rho:V\to\h^*\end{equation}
the unique momentum map for $\rho$ that vanishes at $0$.\footnote{Viewing the symplectic vector space $(V,\si)$ as a symplectic manifold, the representation $\rho$ is a Hamiltonian $G$-action on $V$. Hence it admits a momentum map. In the article \cite{KZ} we used the notation $\mu_V$ instead of $\nu_\rho$. To make the dependence on $\rho$ explicit, I am using $\nu_\rho$ here.} For $a\in G$ and $\phi\in\g^*$ we denote by $a\phi\in T_a^*G$ the image of $\phi$ under the derivative of the left translation by $a$. We define\footnote{In the article \cite{KZ} we used the notation $\wt\mu_D$ for this map. For simplicity I have dropped the tilde here.}\label{mu D rho}
\begin{equation}\label{eq:muD H rho}\muD_{H,\rho}:=\muD_\rho:T^*G\x V\to\h^*,\quad\muD_\rho\big(a,a\phi,v\big):=-\phi|\h+\nu_\rho(v).\end{equation}
This is a momentum map for $\psiD_\rho$. The pair $(Y_\rho,\om_\rho)$ is defined to be the symplectic quotient of $\psiD_\rho$ at 0 \wrt $\muD_\rho$. This means that
\begin{equation}\label{eq:Y H rho}Y_{H,\rho}:=Y_\rho=(\muD_\rho)^{-1}(0)/\psiD_\rho.\end{equation}
(The subgroup $H$ is compact, since it is closed and $G$ is compact. Therefore, the restriction of $\psiD_\rho$ to $(\muD_\rho)^{-1}(0)$ is proper. Since it is also free, the symplectic quotient is well-defined.) The left translation by $G$ on $G$ induces a $G$-action on $T^*G$ and hence on $T^*G\x V$. Since left and right translation commute, this action preserves $(\muD_\rho)^{-1}(0)$ and descends to a $G$-action $\psi_\rho$ on $Y_\rho$, the symplectic quotient of $T^*G\x V$ by the diagonal $H$-action.\footnote{This can be seen as part of symplectic reduction in stages.} This defines $\Model{G}(H,\rho)=\big(Y_\rho,\om_\rho,\psi_\rho)$. 
\item For every $g\in G$ we denote by $R^g:G\to G$\label{R g}, $R^g(a):=ag$, the right translation by $g$, and by $R^g_*:T^*G\to T^*G$ the induced map. The map $\Model{G}$ assigns to every morphism $(g,T):(H,\rho)\to(H',\rho')$ of $\SympRep{G}$ the morphism $\Model{G}(g,T)$ of $\HamEx{G}$ given by
\begin{equation}\label{eq:Model G}\Model{G}(g,T)(y):=\big[R^{g^{-1}}_*(a,a\phi),Tv\big],\end{equation}
where $(a,a\phi,v)$ is an arbitrary representative of $y$. (Here on the right hand side we denote by $\big[a',a'\phi',v'\big]$ the equivalence class of $\big(a',a'\phi',v'\big)$.)
\end{itemize}
The main result is the following. (As always, we assume that $G$ is compact and connected.)
\begin{thm}[Hamiltonian $G$-model functor]\label{thm:model}
\begin{enui}
\item\label{thm:model:well obj}(well-definedness on objects) The map $\Model{G}$ is well-defined on objects, i.e., $\psi_\rho$ is indeed an exact Hamiltonian $G$-action. 
\item\label{thm:model:well mor}(well-definedness on morphisms) The map $\Model{G}$ is well-defined on morphisms, i.e., 
\begin{equation}\label{eq:R g -1 * T}\left(R^{g^{-1}}_*\x T\right)\left((\muD_\rho)^{-1}(0)\right)\sub(\muD_{\rho'})^{-1}(0),\end{equation}
the right hand side of \reff{eq:Model G} does not depend on the choice of a representative $(a,a\phi,v)$, and $\Model{G}(g,T)$ is a morphism of $\HamEx{G}$. 
\item\label{thm:model:funct}(functoriality) The map $\Model{G}$ is a covariant functor.
\item\label{thm:model:inj}(essential injectivity) The map between the sets of isomorphism classes induced by $\Model{G}$ is injective.
\item\label{thm:model:mor}(morphisms) Let $(H,\rho)$ and $(H',\rho')$ be objects of $\SympRep{G}$, and $(g,T),(\hhat g,\hhat T)$ be morphisms between these objects. $\Model{G}$ maps these morphisms to the same morphism if and only if
\begin{equation}\label{eq:wt T rho'}h':=\hhat gg^{-1}\in H',\quad\hhat T=\rho'_{h'}T.\end{equation}
\item\label{thm:model:prop mor} (momentum properness and morphisms) Let $A$ and $A'$ be objects of $\SympRep{G}$ or let them be objects of $\HamEx{G}$, such that $A'$ is momentum proper\footnote{Recall that an object $A$ of $\SympRep{G}$ is a tuple $\big(H,V,\si,\rho\big)$, where $H$ is a closed subgroup of $G$ and $(V,\si,\rho)$ is a symplectic $H$-representation. Viewing $(V,\si)$ as a symplectic manifold, $\rho$ is a Hamiltonian $H$-action. We call $A$ \emph{momentum proper} iff this action is momentum proper, i.e., if every momentum map for $\rho$ is proper.} 
and there exists a morphism from $A$ to $A'$. Then $A$ is momentum proper.
\item\label{thm:model:prop fct} (momentum properness and model functor) An object of $\SympRep{G}$ is momentum proper if and only if its image under $\Model{G}$ is momentum proper.
\item\label{thm:model:surj} (essential surjectivity) Every momentum proper object of $\HamEx{G}$ is isomorphic to an object in the image of $\Model{G}$.
\end{enui}
\end{thm}
\begin{Rmk}
Theorem \ref{thm:model}\reff{thm:model:mor} characterizes the extent to which the functor $\Model{G}$ is faithful.
\end{Rmk}
Theorem \ref{thm:model} puts previous joint work \cite{KZ} with Y.~Karshon into a categorical framework. Namely, part \reff{thm:model:surj} of this theorem (essential surjectivity of $\Model{G}$) was proved in \cite[1.5.~Theorem]{KZ}, without introducing the categorical setup used in the present article. 

The other parts of Theorem \ref{thm:model} will be proved in the next section. The proof of \reff{thm:model:inj} (essential injectivity) is based on Lemma \ref{le:A y}, which provides criteria under which the symplectic quotient representation of the model action $\Model{G}(H,\rho)$ at a given point is isomorphic to $(H,\rho)$. We also use the fact that if two compact subgroups of a Lie group are conjugate to subgroups of each other then they are conjugate to each other. (This follows from Lemma \ref{le:N conj N'} below.)
\begin{Rmk} Naively, in the definition of a morphism of $\SympRep{G}$, one could try to weaken the condition \eqref{eq:gHg -1} to either the condition $c_g(H)\sub H'$ or $c_g(H)\cont H'$. With this modification the model functor would no longer be well-defined on morphisms. (``$\cont$'' is needed in order for \eqref{eq:R g -1 * T} to hold and ``$\sub$'' is needed for the right hand side of \reff{eq:Model G} not to depend on the choice of a representative. See the proof of Theorem \ref{thm:model}\reff{thm:model:well mor} below.)
\end{Rmk}
We denote by \label{SympRepProp G}\label{HamExProp G}
\[\SympRepProp{G},\quad\HamExProp{G}\]
the full subcategories of $\SympRep{G}$ and $\HamEx{G}$ consisting of momentum proper objects. Theorem \ref{thm:model} has the following application.
\begin{cor}[classification of momentum proper exact Hamiltonian actions]\label{cor:class} The functor $\Model{G}$ induces a bijection
\begin{equation}\label{eq:iso class}\big\{\textrm{isomorphism class of }\SympRepProp{G}\big\}\to\big\{\textrm{isomorphism class of }\HamExProp{G}\big\}.
\end{equation}
\end{cor}
We call this bijection the \emph{classifying map (for momentum proper exact Hamiltonian $G$-actions)}.
\begin{Rmks}\begin{itemize}\item It follows from Theorem \ref{thm:model}\reff{thm:model:prop mor} that the isomorphism class of any object of $\SympRepProp{G}$ is its isomorphism class in the bigger category $\SympRep{G}$. A similar remark applies to $\HamExProp{G}$.
\item Corollary \ref{cor:class} classifies all momentum proper exact Hamiltonian $G$-actions up to isomorphism.
\item Surjectivity of the classifying map \eqref{eq:iso class} follows from Theorem \ref{thm:model}\reff{thm:model:surj}, which was proved in joint work with Y.~Karshon \cite[1.5.~Theorem]{KZ}.
\item The inverse of the classifying map \eqref{eq:iso class} is induced by assigning to a Hamiltonian action its symplectic quotient representation at any suitable point, see Proposition \ref{prop:inv} below.
\item In contrast with Corollary \ref{cor:class} the map induced by $\Model{G}$ between the sets of isomorphism classes of $\SympRep{G}$ and $\HamEx{G}$ is not surjective. To see this, let $Q$ be a connected compact manifold of positive dimension, without boundary. We define $\om$ to be the canonical symplectic form on $T^*Q$ and $\psi$ to be the trivial $G$-action on $T^*Q$. 

We claim that the isomorphism class of $(T^*Q,\om,\psi)$ does not lie in the image of $\Model{G}$. To see this, assume that $\big(H,V,\si,\rho\big)$ is an object of $\SympRep{G}$ for which $\psi_\rho$ is trivial. Then $H=G$ and therefore, $Y_\rho$ is canonically diffeomorphic to $V$. If $\big(Y_\rho,\om_\rho,\psi_\rho\big)$ is isomorphic to $\big(T^*Q,\om,\psi\big)$ then it follows that $Q$ is a singleton. This proves the claim.
\item Many classification results are known for Hamiltonian group actions whose complexity is low. (By definition, the complexity is half the dimension of a generic non-empty reduced space. For references see \cite{KZ}.) What makes Corollary \ref{cor:class} special is that it classifies Hamiltonian actions of \emph{arbitrary} complexity.
\end{itemize}
\end{Rmks}
\begin{proof}[Proof of Corollary \ref{cor:class}] By Theorem \ref{thm:model}(\ref{thm:model:well obj},\ref{thm:model:well mor},\ref{thm:model:prop fct}``$\then$'',\ref{thm:model:funct}) the map \eqref{eq:iso class} is well-defined. By Theorem \ref{thm:model}(\ref{thm:model:inj},\ref{thm:model:surj},\ref{thm:model:prop fct}``$\follows$'') the map \eqref{eq:iso class} is bijective. This proves Corollary \ref{cor:class}.
\end{proof}
By considering the extreme case of the full subgroup $H=G$, this corollary implies that the momentum proper Hamiltonian $G$-actions on contractible manifolds are exactly the momentum proper symplectic $G$-representations, up to isomorphism. See Corollary \ref{cor:contr} below. On the other hand, by considering the extreme case in which the vector space $V$ is trivial, using Corollary \ref{cor:class}, we can classify the \emph{critical} momentum proper exact Hamiltonian $G$-actions in terms of transitive $G$-actions on manifolds.

To explain the latter application, we call an object $(M,\om,\psi)$ of $\HamExProp{G}$ \emph{critical} iff $M$ is homotopy equivalent to some closed\footnote{This means compact and without boundary.} manifold of dimension equal to $\dim(M)/2$.
\begin{rmk}[criticality]\label{rmk:crit} By Corollary \ref{cor:class} there exists an object $(H,\rho)=\big(H,V,\si,\rho\big)$ of $\SympRepProp{G}$, such that $\Model{G}(H,\rho)$ is isomorphic to $(M,\om,\psi)$. The manifold part of $\Model{G}(H,\rho)$ is homotopy equivalent to the closed manifold $G/H$ and has dimension $2(\dim G-\dim H)+\dim V\geq2\dim(G/H)$. It follows that $M$ is not homotopy equivalent to any closed manifold of dimension bigger than $\dim(M)/2$. This justifies the terminology \emph{critical}.
\end{rmk}
We denote\label{HamCrit G}
\[\HamCrit{G}:=\textrm{full subcategory of $\HamExProp{G}$ consisting of critical objects}.\]
For every manifold $Q$ we denote by $\om_Q$\label{om Q} the canonical symplectic form on $T^*Q$. We define the \emph{$G$-cotangent functor} $\Tstar{G}$\label{T*G} to be the canonical functor from the category of $G$-actions on manifolds and $G$-equivariant diffeomorphisms to the category of Hamiltonian $G$-actions and $G$-equivariant symplectomorphisms. It takes an object $(Q,\theta)$ to $(T^*Q,\om_Q)$ together with lifted $G$-action $\theta_*$, and a morphism $f:Q\to Q'$ to the lifted map $f_*:T^*Q\to T^*Q'$.

We define $\Trans{G}$\label{Trans G} to be the category whose objects are the transitive smooth $G$-actions on connected closed manifolds and whose morphisms are the $G$-equivariant diffeomorphisms.
\begin{cor}[classification of critical momentum proper exact Hamiltonian actions]\label{cor:class crit} The functor $\Tstar{G}$ induces a bijection
\begin{equation}\label{eq:iso class }\big\{\textrm{isomorphism class of }\Trans{G}\big\}\to\big\{\textrm{isomorphism class of }\HamCrit{G}\big\}.
\end{equation}
\end{cor}
\begin{Rmks}[classification of critical actions, Eliashberg cotangent bundle conjecture]
\begin{itemize}
\item Part of the statement is that $\Tstar{G}$ maps $\Trans{G}$ to $\HamCrit{G}$.
\item The isomorphism class of any object of $\HamCrit{G}$ in $\HamCrit{G}$ is its isomorphism class in the bigger category $\HamExProp{G}$ (or in $\HamEx{G}$). 
\item Corollary \ref{cor:class crit} classifies the critical momentum proper exact Hamiltonian $G$-actions in terms of transitive $G$-actions on manifolds.
\item The \emph{cotangent functor $T^*$} is the canonical functor from the category of connected closed smooth manifolds and diffeomorphisms to the category of symplectic manifolds and symplectomorphisms. It agrees with $\Tstar{\{e\}}$. The Eliashberg cotangent bundle conjecture states that $T^*$ is essentially injective, i.e., it induces an injective map between the sets of isomorphism classes. See \cite[Problem 31, p.~561]{MS}. Very little is known about this conjecture. See \cite{Ab,EKS,ES} for some results.
\item By Corollary \ref{cor:class crit} the restriction of the functor $\Tstar{G}$ to the category $\Trans{G}$ of transitive $G$-actions is essentially injective. This proves an equivariant version of the Eliashberg cotangent bundle conjecture. In fact, Corollary \ref{cor:class crit} provides more information, namely it also specifies the image of the class of objects of $\Trans{G}$ under $\Tstar{G}$, up to isomorphism.
\item The philosophy behind this application is that symmetry makes problems more accessible. In the present situation it allows for a classification of the structures at hand (transitive $G$-actions and critical Hamiltonian $G$-actions). The same philosophy was for example used recently in \cite{FPP}, where the authors used Delzant's classification of symplectic toric manifolds to prove that certain equivariant symplectic capacities are (dis-)continuous. (Without symmetry the question whether a given symplectic capacity is continuous is hard in general.)
\end{itemize}
\end{Rmks}
We will prove Corollary \ref{cor:class crit} in Section \ref{sec:proof:cor:class crit}.\\

As another application of Corollary \ref{cor:class}, we now classify the momentum proper Hamiltonian $G$-actions on contractible manifolds. Here we consider another extreme case, in which the subgroup $H$ equals $G$. We denote by $\symprep{G}$\label{symprep G} the category whose objects are symplectic $G$-representations and whose morphisms are $G$-equivariant linear symplectic maps (possibly not surjective), and by\label{symprepprop G}
\[\symprepprop{G}\]
the full subcategory consisting of momentum proper objects. We denote by $\HamContr{G}$\label{HamContr G} the full subcategory of $\HamEx{G}$ consisting of those objects $(M,\om,\psi)$ for which $M$ is contractible, and by\label{HamContrProp G}
\[\HamContrProp{G}\]
the full subcategory of $\HamContr{G}$ consisting of momentum proper objects. We denote by\label{I G}
\[\I{G}:\symprep{G}\to\HamContr{G},\quad\IProp{G}:\symprepprop{G}\to\HamContrProp{G}\]
the inclusion functor and its restriction to the momentum proper subcategories. We denote by $\I{G}_*,\IProp{G}_*$ the maps between the sets of isomorphism classes induced by $\I{G},\IProp{G}$. 
\begin{rmks}\label{rmk:iso}\begin{enui}\item\label{rmk:iso:iso class} The isomorphism class of any object in $\symprepprop{G}$ is its isomorphism class in the bigger category $\symprep{G}$. This follows from Remark \ref{rmk:sympl quot Phi}\reff{rmk:sympl quot Phi:Ham} below. Similar remarks apply to the subcategory $\HamContrProp{G}$ of $\HamContr{G}$ and the subcategory $\HamContr{G}$ of $\HamEx{G}$.
\item\label{rmk:iso:I G *} The map $\I{G}_*$ extends the map $\IProp{G}_*$. By \reff{rmk:iso:iso class} this statement makes sense.
\end{enui}
\end{rmks}
\begin{cor}[classification of momentum proper Hamiltonian actions on contractible manifolds]\label{cor:contr} \begin{enui}\item\label{cor:contr:inj} The map $\I{G}_*$ is injective.
\item\label{cor:contr:surj} The map $\IProp{G}_*$ is surjective.
\end{enui}
\end{cor}
\begin{Rmks}\begin{itemize}
\item It follows from part \reff{cor:contr:inj} of this corollary and Remark \ref{rmk:iso}\reff{rmk:iso:I G *} that $\IProp{G}_*$ is injective. Using \reff{cor:contr:surj}, this map is bijective.
\item Part \reff{cor:contr:surj} means that every momentum proper Hamiltonian $G$-action on a contractible symplectic manifold is symplectically linearizable\footnote{We call a symplectic $G$-action $\psi$ on a symplectic manifold $(M,\om)$ \emph{symplectically linearizable} 
iff there exist a symplectic $G$-representation $(V,\si,\rho)$ and a symplectomorphism between $V$ and $M$ that intertwines $\rho$ and $\psi$.}.
\item The statement of Corollary \ref{cor:contr} means that the momentum proper Hamiltonian $G$-actions on contractible symplectic manifolds agree with the momentum proper symplectic $G$-representations, up to isomorphism. This classifies these actions.
\item Assume that $G$ is non-abelian. In contrast with part \reff{cor:contr:surj} the map $\I{G}_*$ is not surjective. This follows from \cite[Corollary 8.4]{KZ}.
\end{itemize}
\end{Rmks}
For the proof of Corollary \ref{cor:contr}\reff{cor:contr:surj} we need the following.
\begin{rmk}\label{rmk:V Y rho} For every symplectic $G$-representation $(V,\si,\rho)$ the map
\begin{equation}\label{eq:A rho V Y rho}\II{G}{\rho}:V\to Y_\rho,\quad\II{G}{\rho}(v):=[e,0,v],\end{equation}
is a $G$-equivariant symplectomorphism, i.e., an isomorphism from $\I{G}(\rho)=\rho$ to $\Model{G}(G,\rho)$ in $\HamEx{G}$. This follows from a straightforward argument.
\end{rmk}
\begin{proof}[Proof of Corollary \ref{cor:contr}] \textbf{\reff{cor:contr:inj}:} Let $R$ and $R'$ be isomorphism classes of $\symprep{G}$ that are mapped to the same class under $\I{G}_*$. We choose representatives $(V,\si,\rho),(V',\si',\rho')$ of $R,R'$ and an isomorphism $\Phi$ in $\HamContr{G}$ from $\I{G}(\rho)$ to $\I{G}(\rho')$. The differential $d\Phi(0):T_0V\to T_{\Phi(0)}V'$ is an isomorphism from $d\rho(0)$ to $d\rho'(\Phi(0))$ in $\symprep{G}$. Since $\rho$ is linear, the canonical identification between $V$ and $T_0V$ is an isomorphism from $\rho$ to $d\rho(0)$ in $\symprep{G}$. Similarly, $\rho'$ is isomorphic to $d\rho'(\Phi(0))$. Combining these three isomorphisms, it follows that $\rho$ and $\rho'$ are isomorphic in $\symprep{G}$, i.e., $R=R'$. Hence the map $\I{G}_*$ is injective. This proves \reff{cor:contr:inj}.\\

\textbf{\reff{cor:contr:surj}:} Let $\Psi$ be an isomorphism class of objects of $\HamContrProp{G}$. We choose a representative $(M,\om,\psi)$ of $\Psi$. By Theorem \ref{thm:model}\reff{thm:model:surj} there exists an object $(H,\rho)$ of $\SympRep{G}$, such that $\psi_\rho:=\Model{G}(H,\rho)$ is isomorphic to $\psi$ in $\HamEx{G}$. By Theorem \ref{thm:model}\reff{thm:model:prop mor} $\psi_\rho$ is momentum proper. Hence by Theorem \ref{thm:model}\reff{thm:model:prop fct}``$\follows$'' $(H,\rho)$ is momentum proper. Therefore, $\IProp{G}(\rho):=\I{G}(\rho)$ is well-defined.

Since $M$ is contractible, the same holds for $Y_\rho$. Therefore, by the proof of \cite[7.6 Lemma]{KZ} we have $H=G$. Hence by Remark \ref{rmk:V Y rho} $\IProp{G}(\rho)$ and $\psi_\rho$ are isomorphic in $\HamEx{G}$ and hence in $\HamContrProp{G}$. It follows that $\IProp{G}(\rho)$ and $\psi$ are isomorphic in $\HamContrProp{G}$. Hence $\IProp{G}_*([\rho])=\Psi$. Thus $\IProp{G}_*$ is surjective. This proves \reff{cor:contr:surj} and completes the proof of Corollary \ref{cor:contr}.
\end{proof}
\begin{Rmks} 
\begin{itemize}
\item (This remark will be used in the next one.) We define $\SYMPREP{G}$\label{SYMPREP G} to be the category with objects the symplectic $G$-representations and morphisms between $\rho,\rho'$ given by pairs $(g,T)$, where $g\in G$ and $T:V\to V'$ is a linear symplectic map, such that \eqref{eq:T rho h} holds. The composition is defined by \eqref{eq:g' T' g T}. We define the functor\label{III G}
\[\III{G}:\SYMPREP{G}\to\SympRep{G},\quad\III{G}(\rho):=(G,\rho),\quad\III{G}=\textrm{identity on morphisms.}\]
We may view $\SYMPREP{G}$ as a full subcategory of $\SympRep{G}$ via this functor. We define the map\label{F G}
\[\F{G}:\SYMPREP{G}\to\symprep{G},\quad\F{G}=\textrm{identity on objects,}\quad\F{G}(g,T):=\rho'_{g^{-1}}T.\]
A straightforward argument shows that this map is a covariant functor.
\item Part \reff{cor:contr:inj} of Corollary \ref{cor:contr} can alternatively be deduced from Theorem \ref{thm:model}\reff{thm:model:inj} as follows. Let $R,R'$ be isomorphism classes of $\symprep{G}$ that are mapped to the same class under $\I{G}_*$. We choose representatives $\rho,\rho'$ of $R,R'$. Then $\I{G}(\rho)$ and $\I{G}(\rho')$ are isomorphic. Using Remark \ref{rmk:V Y rho}, it follows that $\Model{G}\circ\III{G}(\rho)$ and $\Model{G}\circ\III{G}(\rho')$ are isomorphic. Hence by Theorem \ref{thm:model}\reff{thm:model:inj} there exists an isomorphism $(g,T)$ in $\SympRep{G}$ from $\III{G}(\rho)=(G,\rho)$ to $\III{G}(\rho')=(G,\rho')$. It follows that $(g,T)$ is an isomorphism in $\SYMPREP{G}$ from $\rho$ to $\rho'$. Therefore, $\F{G}(g,T)=\rho'_{g^{-1}}T$ is an isomorphism in $\symprep{G}$ from $\F{G}(\rho)=\rho$ to $\F{G}(\rho')=\rho'$. It follows that $R=[\rho]=[\rho']=R'$. This shows that $\I{G}_*$ is injective, i.e., part \reff{cor:contr:inj} of Corollary \ref{cor:contr}.
\item A straightforward argument shows that the map $\II{G}{}:\rho\mapsto\II{G}{\rho}$ is a natural isomorphism between the functors $\I{G}\circ\F{G}$ and $\Model{G}\circ\III{G}$,
\[\I{G}\circ\F{G}\overset{\II{G}{}}{\underset{\simeq}{\longrightarrow}}\Model{G}\circ\III{G}.\]
This means that for every morphism $(g,T):\rho\to\rho'$ of $\SYMPREP{G}$ the diagram 
\[\label{eq:nat}\begin{CD}
\I{G}\circ\F{G}(\rho) @> \I{G}\circ\F{G}(g,T) >> \I{G}\circ\F{G}(\rho') \\
@V \II{G}{\rho} VV @VV \II{G}{\rho'} V  \\
\Model{G}\circ\III{G}(\rho) @> \Model{G}\circ\III{G}(g,T) >> \Model{G}\circ\III{G}(\rho').
\end{CD}\]
commutes, and that $\II{G}{\rho}$ is an isomorphism for every object $\rho$ of $\SYMPREP{G}$. In other words the map $\Model{G}(g,T)$ is given by
\[\Model{G}(g,T)=\F{G}(g,T)=\rho'_{g^{-1}}T:Y_\rho\to Y_{\rho'}\]
via the natural identifications $\II{G}{\rho}:V\overset{\cong}{\to}Y_\rho$ and $\II{G}{\rho'}:V'\overset{\cong}{\to}Y_{\rho'}$.
\end{itemize}
\end{Rmks}

\section{Proof of Theorem \ref{thm:model}(\ref{thm:model:well obj}-\ref{thm:model:prop fct}) 
(Hamiltonian $G$-model functor)}\label{sec:proof:thm:model}

For the proof of Theorem \ref{thm:model}\reff{thm:model:well obj} we need the following. We denote by $\Ad$ and $\Ad^*$\label{Ad Ad*} the adjoint and coadjoint representations of $G$. We define the map\label{mu L}
\[\muL:T^*G\to\g^*,\quad\muL(a,a\phi)=\Ad^*(a)\phi.\]
This is a momentum map for the lifted left-translation action of $G$ on $T^*G$. We denote by $\pr_1:T^*G\x V\to T^*G$\label{pr1} the canonical projection. Since left and right translations commute, $\muL$ is preserved by the lifted right translation action of $H$ on $T^*G$. Hence the map $\muL\circ\pr_1$ descends to a map\footnote{In the article \cite{KZ} we used the notation $\mu_Y$ instead of $\mu_\rho$. I am using $\mu_\rho$ here to make the dependence on $\rho$ explicit and to stay in line with the notation $Y_\rho,\om_\rho$.}\label{mu rho}
\[\mu_\rho:Y_\rho\to\g^*.\]
\begin{proof}[Proof of Theorem \ref{thm:model}\reff{thm:model:well obj}] The map $\mu_\rho$ is a momentum map for $\psi_\rho$. Hence $\psi_\rho$ is a Hamiltonian action, and therefore $\Model{G}$ is well-defined on objects, as claimed.
\end{proof}
For the proof of Theorem \ref{thm:model}\reff{thm:model:well mor} we need the following.
\begin{rmk}[product of proper maps]\label{rmk:prod proper} Let $X,Y,X',Y'$ be topological spaces, with $Y$ and $Y'$ Hausdorff. Let $f:X\to Y$ and $f':X'\to Y'$ be proper continuous maps. Then the Cartesian product map $f\x f':X\x X'\to Y\x Y'$ is proper. This follows from an elementary argument. (Hausdorffness ensures that every compact subset of $Y\x Y'$ is closed.)
\end{rmk}
\begin{proof}[Proof of Theorem \ref{thm:model}\reff{thm:model:well mor}] Let $\big(H,V,\si,\rho\big)$ and $\big(H',V',\si',\rho'\big)$ be objects of $\SympRep{G}$ and $(g,T)$ a morphism between them. We denote by $\h$ and $\h'$ the Lie algebras of $H$ and $H'$. By \eqref{eq:gHg -1} we have $c_{g^{-1}}(H')=H$. It follows that $\Ad_{g^{-1}}(\h')=\h$. Hence $\Ad^*(g)=\Ad_{g^{-1}}^*$ induces a map from $\h^*$ to ${\h'}^*$, which we again denote by $\Ad^*(g)$. We have
\begin{equation}\label{eq:Ad * g phi}\Ad^*(g)(\phi)|\h'=\Ad^*(g)(\phi|\h),\quad\forall(a,a\phi)\in T^*G.
\end{equation}
The map 
\[\rho'\circ c_g:H\to\big\{\textrm{isomorphisms of }(V',\si')\big\}\]
is a Hamiltonian action with momentum map
\[c_g^*\circ\nu_{\rho'}=\Ad_g^*\circ\nu_{\rho'}:V'\to\h^*,\]
where $\nu_{\rho'}$ is as in \eqref{eq:nu rho}. By \eqref{eq:T rho h} $\rho'$ leaves the image $T(V)$ invariant and $T$ is a symplectic embedding that is equivariant \wrt $\rho$ and $\rho'\circ c_g$.
It follows that
\begin{equation}\label{eq:Ad g * nu rho'}\Ad_g^*\circ\nu_{\rho'}\circ T=\nu_\rho.\end{equation}
(Here we use that both sides vanish at $v=0\in V$.) For every $(a,a\phi,v)\in T^*G\x V$ we have
\begin{align*}\muD_{\rho'}\circ\left(R^{g^{-1}}_*\x T\right)(a,a\phi,v)&=\muD_{\rho'}\big(ag^{-1},ag^{-1}\Ad^*(g)(\phi),Tv\big)\\
&=-\Ad^*(g)(\phi)|\h'+\nu_{\rho'}(Tv)\\
&=\Ad^*(g)\big(-\phi|\h+\nu_\rho(v)\big)\qquad\textrm{(using (\ref{eq:Ad * g phi},\ref{eq:Ad g * nu rho'}))}\\
&=\Ad^*(g)\circ\muD_\rho(a,a\phi,v).
\end{align*}
The claimed inclusion \eqref{eq:R g -1 * T} follows. We define
\[\wt\Phi:=R^{g^{-1}}_*\x T:(\muD_\rho)^{-1}(0)\to(\muD_{\rho'})^{-1}(0).\]
Let $h\in H$. By \eqref{eq:gHg -1} we have $h':=c_g(h)\in H'$. By \eqref{eq:T rho h} the map $\wt\Phi$ intertwines the diagonal action of $h$ on $T^*G\x V$ with the diagonal action of $h'$ on $T^*G\x V'$. It follows that the right hand side of \eqref{eq:Model G} does not depend on the choice of the representative $(a,a\phi,v)$, as claimed. We denote by 
\[\Phi:=\Model{G}(g,T):Y_\rho\to Y_{\rho'},\]
the map induced by $\wt\Phi$. We show that $\Phi$ is a morphism of $\HamEx{G}$. The map $\wt\Phi$ is smooth, presymplectic, and equivariant \wrt the $G$-actions induced by the left translations on $G$. It follows that $\Phi$ is smooth, symplectic, and equivariant \wrt to the $G$-actions $\psi_\rho$ and $\psi_{\rho'}$.

\begin{claim}\label{claim:T Phi proper} The maps $T$ and $\Phi$ are proper.
\end{claim}
\begin{proof}[Proof of Claim \ref{claim:T Phi proper}] The map $T:V\to V'$ is linear symplectic and hence injective. Since $V$ is finite-dimensional, it follows that
\[\sup_{0\neq v\in V}\frac{\Vert v\Vert}{\Vert Tv\Vert'}<\infty,\]
where $\Vert\cdot\Vert,\Vert\cdot\Vert'$ are arbitrary norms on $V,V'$. This implies that $T$ is proper, as claimed.

We denote by
\begin{equation}\label{eq:pi rho}\pi_\rho:(\muD_\rho)^{-1}(0)\to Y_\rho=(\muD_\rho)^{-1}(0)/\psiD_\rho\end{equation}
the canonical projection. Let $K'\sub Y_{\rho'}$ be a compact subset. Since $\Phi\circ\pi_\rho=\pi_{\rho'}\circ\wt\Phi$, we have
\begin{equation}\label{eq:pi Phi}\pi_\rho^{-1}\circ\Phi^{-1}(K')=\wt\Phi^{-1}\circ\pi_{\rho'}^{-1}(K').\end{equation}
The projection $\pi_{\rho'}$ is proper, since $H'$ is compact. It follows that $\pi_{\rho'}^{-1}(K')$ is compact. The map $R^{g^{-1}}_*:T^*G\to T^*G$ is proper, since it is invertible with continuous inverse. Using Remark \ref{rmk:prod proper} and properness of $T$, it follows that the Cartesian product map $R^{g^{-1}}_*\x T:T^*G\x V\to T^*G\x V'$ is proper. Since this map restricts to $\wt\Phi$ on $(\muD_\rho)^{-1}(0)$, it follows that $\wt\Phi$ is proper. Since $\pi_{\rho'}^{-1}(K')$ is compact, it follows that the right hand side of \eqref{eq:pi Phi} is compact, hence also the left hand side. Since $\pi_\rho$ maps this set to $\Phi^{-1}(K')$, it follows that $\Phi^{-1}(K')$ is compact. This proves Claim \ref{claim:T Phi proper}.
\end{proof}

Using Claim \ref{claim:T Phi proper}, it follows that $\Phi$ is a $G$-equivariant proper symplectic embedding, i.e., a morphism of $\HamEx{G}$. This proves that the map $\Model{G}$ is well-defined on morphisms. This completes the proof of Theorem \ref{thm:model}\reff{thm:model:well mor}.
\end{proof}

\begin{proof}[Proof of Theorem \ref{thm:model}\reff{thm:model:funct}] It follows from a straightforward argument that $\Model{G}$ maps the unit morphisms to unit morphisms and intertwines the compositions. Hence it is a covariant functor. This proves Theorem \ref{thm:model}\reff{thm:model:funct}.
\end{proof}

For the proof of Theorem \ref{thm:model}\reff{thm:model:inj} we need the following. Let $G$ be a group, $X$ a set, $\psi$ an action of $G$ on $X$, and $x\in X$. We denote by\label{G x}
\[G_x:=\Stab{\psi}{x}:=\big\{g\in G\,\big|\,\psi_g(x)=x\big\}\]
the stabilizer of $x$ under $\psi$.
\begin{rmk}\label{rmk:stab} Let $G$ be a Lie group, $(\rho,H)$ an object of $\SympRep{G}$, and $y=[a,a\phi,v]\in Y_\rho$. Then
\[G_y=\big\{c_a(h)\,\big|\,h\in H:\,\rho_hv=v\big\}.\]
\end{rmk}

\begin{lemma}\label{le:N conj N'} Let $G$ be a topological (finite-dimensional) manifold with a continuous group structure, $N,N'$ compact submanifolds of $G$, and $g\in G$, such that
\begin{equation}\label{eq:g N g -1 N'}c_g(N)\sub N',\end{equation}
and $N'$ is conjugate to some subset of $N$. Then we have
\[c_g(N)=N'.\]
\end{lemma}

In the proof of this lemma we will use the following.
\begin{rmk}[invariance of domain]\label{rmk:inv dom} Let $M$ and $N$ be topological manifolds of the same finite dimension, without boundary. Then every continuous injective map from $M$ to $N$ is open. In the case $M=N=\R^n$ this is the statement of the Invariance of Domain Theorem, see \cite[Theorem 2B.3, p.~172]{Ha}. The general situation can be reduced to this case.
\end{rmk}

\begin{proof}[Proof of Lemma \ref{le:N conj N'}] We choose $g'\in G$, such that
\begin{equation}\label{eq:g'N'g'}c_{g'}(N')\sub N,\end{equation}
and define $\psi:=c_{g'g}:G\to G$. We have
\[\psi(N)=c_{g'}\circ c_g(N)\sub N.\]
Let $A$ be a connected component of $N$. Since $N$ is a submanifold of $G$, the set $A$ is open in $N$. The map $\psi$ is bijective and continuous. Hence by Remark \ref{rmk:inv dom} the restriction $\psi:N\to N$ is open. Thus $\psi(A)$ is open in $N$. 

Since $A$ is a connected component of $N$, it is closed in $N$. Since $N$ is compact, it follows that $A$ is compact. Therefore, $\psi(A)$ is compact and hence a closed subset of $N$. It follows that $\psi(A)$ is a connected component of $N$. Hence the map
\begin{equation}\label{eq:conn comp}\big\{\textrm{connected component of }N\big\}\ni A\mapsto\psi(A)\in\big\{\textrm{connected component of }N\big\}\end{equation}
is well-defined. This map is injective. Since $N$ is compact, the number of its connected components is finite. It follows that the map \eqref{eq:conn comp} is surjective. It follows that $N\sub\psi(N)$, and therefore, $c_{g'}^{-1}(N)\sub c_g(N)$. By \eqref{eq:g'N'g'} we have $N'\sub c_{g'}^{-1}(N)$. It follows that $N'\sub c_g(N)$. Combining this with \eqref{eq:g N g -1 N'}, it follows that $c_g(N)=N'$. This proves Lemma \ref{le:N conj N'}.
\end{proof}

Let $G$ be a Lie group, $(M,\om,\psi)$ a symplectic $G$-action, and $x\in M$.
\begin{Rmk} The isotropy representation of $\psi$ at $x$ is by definition the map\label{isotr psi x}
\[\isotr{\psi}{x}:\Stab{\psi}{x}\x T_xM\to T_xM,\quad(g,v)\mapsto d\psi_g(x)v.\]
This is a symplectic representation of the isotropy group $\Stab{\psi}{x}$.
\end{Rmk}
In order to define the symplectic quotient representation of $\psi$ at $x$, we need the following remarks. 
\begin{rmks}[symplectic quotient representation]\label{rmk:lin quot}\begin{enui}
\item\label{rmk:lin quot:d psi g im L x} Let $G$ be a Lie group, $(M,\psi)$ a $G$-action on a manifold, and $x\in M$. We denote by $\g$ the Lie algebra of $G$ and by
\begin{equation}\label{eq:L psi x}L_x:=L^\psi_x:\g\to T_xM\end{equation}
the infinitesimal action at $x$. The equality
\[d\psi_g(x)(\im L_x)=\im L_{\psi_g(x)}\]
holds.
\item\label{rmk:lin quot:quot} Let $(V,\si)$ be a symplectic vector space and $W\sub V$ a linear space. We denote by\label{W si}
\[W^\si:=\big\{v\in V\,\big|\,\si(v,w)=0,\,\forall w\in W\big\}\]
the symplectic complement of $W$. Let $(M,\om,\psi)$ be a symplectic $G$-action and $x\in M$. The form $\om_x$ induces a linear symplectic form $\BAR\om_x$\label{BAR om x} on the quotient space
\begin{equation}\label{eq:V psi x}V^\psi_x:=(\im L_x)^{\om_x}/\big(\im L_x\cap(\im L_x)^{\om_x}\big).\end{equation}
It follows from \reff{rmk:lin quot:d psi g im L x} that $d\psi_g(x)\left((\im L_x)^{\om_x}\right)=(\im L_{\psi_g(x)})^{\om_{\psi_g(x)}}$. Therefore, using \reff{rmk:lin quot:d psi g im L x} again, $d\psi_g(x)$ induces a map
\begin{equation}\label{eq:V x}V^\psi_x\to V^\psi_{\psi_g(x)}.\end{equation}
This map is a linear symplectic isomorphism \wrt $\BAR\om_x$ and $\BAR\om_{\psi_g(x)}$. 
\end{enui}
\end{rmks}
We define the \emph{symplectic quotient representation of $\psi$ at $x$} to be the map\label{quot psi x}
\begin{equation}\label{eq:quot}\quot{\psi}{x}:\Stab{\psi}{x}\x V^\psi_x\to V^\psi_x,\end{equation}
where $\quot{\psi}{x}(g,\cdot)$ is given by the map \eqref{eq:V x}. By Remark \ref{rmk:lin quot}\reff{rmk:lin quot:quot} this is a well-defined symplectic representation of $\Stab{\psi}{x}$ on the linear symplectic quotient $V^\psi_x$ of $(\im L_x)^{\om_x}$.\footnote{In the literature $\quot{\psi}{x}$ is called ``symplectic slice representation''. This terminology seems misleading, since $\quot{\psi}{x}$ does not involve any choice of a local slice.} 
\begin{rmks}[equivariant symplectomorphism, symplectic quotient representations]\label{rmk:sympl quot Phi} Let $G$ be a Lie group, $(M,\om,\psi),(M',\om',\psi')$ symplectic $G$-actions, $\Phi:M\to M'$ a $G$-equivariant symplectomorphism, $x\in M$, and $x':=\Phi(x)$.
\begin{enui}\item\label{rmk:sympl quot Phi:quot} Since $\Phi$ is $G$-equivariant and injective, we have
\[\Stab{\psi}{x}=\Stab{\psi'}{x'}.\]
Furthermore, we have $d\Phi(x)L^\psi_x=L^{\psi'}_{x'}$, and therefore, $d\Phi(x)\big(\im L^\psi_x\big)=\im L^{\psi'}_{x'}$. Since $\Phi$ is symplectic, it follows that $d\Phi(x)$ induces a map
\[V^\psi_x\to V^{\psi'}_{x'}.\]
This map is a linear symplectic isomorphism that intertwines $\quot{\psi}{x}$ and $\quot{\psi'}{x'}$.
\item\label{rmk:sympl quot Phi:Ham} If $\psi'$ is Hamiltonian with momentum map $\mu'$ then $\mu'\circ\Phi$ is a momentum map for $\psi$.
\end{enui}
\end{rmks}
\begin{lemma}[symplectic quotient representation for model action]\label{le:A y} Let $G$ be a compact Lie group and $(H,V,\si,\rho)$ an object of $\SympRep{G}$. We denote
\[\big(Y_\rho,\om_\rho,\psi_\rho\big):=\Model{G}(H,\rho).\]
Let $y\in Y_\rho$ be a point for which $\mu_\rho(y)$ is central and $\Stab{\psi_\rho}{y}=c_a(H)$, for some representative $(a,a\phi,v)$ of $y$. Then $\rho=(H,\rho)$ is isomorphic to $\quot{\psi_\rho}{y}$.
\end{lemma}
\begin{Rmk} The subgroup $c_a(H)$ does not depend on the choice of the representative $(a,a\phi,v)$ of $y$.
\end{Rmk}
In the proof of this lemma we will use the following.
\begin{rmk}[momentum map]\label{rmk:ker d mu} Let $(M,\om,\psi)$ be a Hamiltonian $G$-action, $\mu$ a momentum map for $\psi$, and $x\in M$. Then
\[\ker d\mu(x)=(\im L^\psi_x)^{\om_x}.\]
\end{rmk}
\begin{proof}[Proof of Lemma \ref{le:A y}] We choose a representative $\wt y:=(a,a\phi,v)$ of $y$. We define\label{iota a phi}
\[\iota_{a,\phi}:V\to T^*G\x V,\quad\iota_{a,\phi}(w):=(a,a\phi,w).\]
\begin{claim}\label{claim:im d iota phi}
\begin{equation}\label{eq:im d iota phi}\im\big(d\iota_{a,\phi}(v)\big)\sub\ker d\muD_\rho(\wt y).\end{equation}
\end{claim}
\begin{proof}[Proof of Claim \ref{claim:im d iota phi}] By our hypothesis $\mu_\rho(y)=\muL(\wt y)=\Ad^*(a)\phi$ is a central element of $\g^*$. Hence, for every $g\in G$, we have
\[\Ad^*(a)\phi=\Ad^*(c_a(g))\Ad^*(a)\phi=\Ad^*(a)\Ad^*(g)\phi,\]
and therefore $\phi=\Ad^*(g)\phi$. Hence $\phi$ is a central element of $\g^*$. For every $h\in H$, we have
\begin{align*}[a,a\phi,v]&=y\\
&=\psi_\rho(c_a(h),y)\qquad\textrm{(since $\Stab{\psi_\rho}{y}=c_a(H)$)}\\
&=\big[c_a(h)a,c_a(h)a\phi,v\big]\\
&=\big[ah,a\phi h,v\big]\qquad\textrm{(using that $\phi$ is central)}\\
&=\big[a,a\phi,\rho_hv\big],
\end{align*}
and therefore $\rho_hv=v$. Hence $v$ is a fixed point of $\rho$. It follows that $d\nu_\rho(v)=0$. Since $\muD_\rho(a,a\phi,w)=-\phi|\h+\nu_\rho(w)$, it follows that 
\[d(\muD_\rho\circ\iota_{a,\phi})(v)=d\nu_\rho(v)=0.\]
The inclusion \eqref{eq:im d iota phi} follows. This proves Claim \ref{claim:im d iota phi}.
\end{proof}
We define $\pi_\rho$ as in \eqref{eq:pi rho}, and
\[A^\rho_{\wt y}:=d\pi_\rho(\wt y)d\iota_{a,\phi}(v):V\to T_yY_\rho.\]
By Claim \ref{claim:im d iota phi} this map is well-defined.
\begin{claim}\label{claim:A rho wt y} The pair $\big(a,A^\rho_{\wt y}\big)$ is a morphism from $\rho$ to $\isotr{\psi_\rho}{y}$ (the isotropy representation of $\psi_\rho$ at $y$).
\end{claim}
\begin{proof}[Proof of Claim \ref{claim:A rho wt y}] The map $\iota_{a,\phi}$ is a symplectic embedding. It follows that $A^\rho_{\wt y}$ is linear symplectic. We denote by $\psiL:G\x T^*G\x V\to T^*G\x V$ the action induced by the left-translation on $G$. Let $h\in H$. For all $w\in V$, we have
\begin{align*}\iota_{a,\phi}\circ\rho_h(w)&=\big(a,a\phi,\rho_hw\big)\\
&=\big(ah,a\phi h,w\big)\\
&=(\psiL_\rho)_{c_a(h)}\circ\iota_{a,\phi}(w)\qquad\textrm{(using that $\phi$ is central)}
\end{align*}
Using that $\rho_h$ is linear, it follows that
\begin{align*}d\iota_{a,\phi}(v)\rho_h&=d\iota_{a,\phi}(v)d\rho_h(v)\\
&=d\big((\psiL_\rho)_{c_a(h)}\big)(\wt y)d\iota_{a,\phi}(v).
\end{align*}
Since $\pi_\rho\circ(\psiL_\rho)_g=(\psi_\rho)_g\circ\pi_\rho$, it follows that
\begin{align*}A^\rho_{\wt y}\rho_h&=d\pi_\rho(\wt y)d\iota_{a,\phi}(v)\rho_h\\
&=d(\psi_\rho)_{c_a(h)}(y)d\pi_\rho(\wt y)d\iota_{a,\phi}(v)\\
&=d(\psi_\rho)_{c_a(h)}(y)A^\rho_{\wt y}.
\end{align*}
The statement of Claim \ref{claim:A rho wt y} follows.
\end{proof}
Let $y\in Y_\rho$. Recall that
\[L_y=L^{\psi_\rho}_y:\g\to T_yY_\rho\]
denotes the infinitesimal $\psi_\rho$-action. 
\begin{claim}\label{claim:A L}
\begin{eqnarray}\label{eq:isotropic}&\im L_y\textrm{ is isotropic,}&\\
\label{eq:sub}&\im A^\rho_{\wt y}\sub\big(\im L_y\big)^{(\om_\rho)_y}.&
\end{eqnarray}
\end{claim}
\begin{proof}[Proof of Claim \ref{claim:A L}] Proof of \eqref{eq:isotropic}: Our hypothesis that $\mu_\rho(y)$ is central implies that
\[\im L_y\sub\ker d\mu_\rho(y).\]
By Remark \ref{rmk:ker d mu} we have
\begin{equation}\label{eq:ker d mu rho y}\ker d\mu_\rho(y)=(\im L_y)^{(\om_\rho)_y}.\end{equation}
Statement \eqref{eq:isotropic} follows.

Proof of \eqref{eq:sub}: Since $\mu_\rho\circ\pi_\rho=\muL\circ\pr_1$, we have
\begin{align*}d\mu_\rho(y)A^\rho_{\wt y}&=d\mu_\rho(y)d\pi_\rho(\wt y)d\iota_{a,\phi}(v)\\
&=d\big(\mu_\rho\circ\pi_\rho\circ\iota_{a,\phi}\big)(v)\\
&=d\big(\muL\circ\pr_1\circ\iota_{a,\phi}\big)(v)\\
&=0.
\end{align*}
Here in the last step we used that the map $\pr_1\circ\iota_{a,\phi}$ is constantly equal to $(a,a\phi)$. It follows that
\[\im A^\rho_{\wt y}\sub\ker d\mu_\rho(y).\]
Using \eqref{eq:ker d mu rho y}, the claimed inclusion \eqref{eq:sub} follows. This completes the proof of Claim \ref{claim:A L}.
\end{proof}
By part \eqref{eq:isotropic} of this claim there is a canonical projection
\[\pr^\rho_y:\big(\im L_y\big)^{(\om_\rho)_y}\to V^{\psi_\rho}_y=\big(\im L_y\big)^{(\om_\rho)_y}/\im L_y.\]
By part \eqref{eq:sub} the restriction
\[\pr^\rho_y\big|\im A^\rho_{\wt y}\]
is well-defined. It follows from Claim \ref{claim:A rho wt y} and the equality $\Stab{\psi_\rho}{y}=c_a(H)$ that $\im A^\rho_{\wt y}$ is invariant under $\isotr{\psi_\rho}{y}$.
\begin{claim}\label{claim:sympl iso} The pair $\left(e,\pr^\rho_y\big|\im A^\rho_{\wt y}\right)$ is an isomorphism between the restriction of $\isotr{\psi_\rho}{y}$ to $\im A^\rho_{\wt y}$ and $\quot{\psi_\rho}{y}$.
\end{claim}
\begin{proof}[Proof of Claim \ref{claim:sympl iso}] The projection $\pr^\rho_y$ is presymplectic. Since $\im A^\rho_{\wt y}$ is symplectic, the restriction $\pr^\rho_y\big|\im A^\rho_{\wt y}$ is linear symplectic and therefore injective. We have
\begin{align*}\dim\left(V^{\psi_\rho}_y=\big(\im L_y\big)^{(\om_\rho)_y}/\im L_y\right)&=\dim(Y_\rho)-2\dim\im L_y\\
&=\dim(T^*G\x V)-2\dim H-2\dim\im L_y\\
&=2\dim G+\dim V-2\dim H-2\dim G+2\dim\Stab{\psi_\rho}{y}\\
&=\dim V\qquad\textrm{(since $\Stab{\psi_\rho}{y}=c_a(H)$)}\\
&=\dim\im A^\rho_{\wt y}\qquad\textrm{(since $A^\rho_{\wt y}$ is linear symplectic, hence injective)}\\
&=\dim\left(\pr^\rho_y\left(\im A^\rho_{\wt y}\right)\right)\qquad\textrm{(since $\pr^\rho_y\big|\im A^\rho_{\wt y}$ is injective).}
\end{align*}
It follows that $V^{\psi_\rho}_y=\pr^\rho_y\left(\im A^\rho_{\wt y}\right)$, hence $\pr^\rho_y\big|\im A^\rho_{\wt y}$ is surjective. Hence this map is a linear symplectic isomorphism. It is $\Stab{\psi_\rho}{y}$-equivariant. The statement of Claim \ref{claim:sympl iso} follows.
\end{proof}
It follows from Claims \ref{claim:A rho wt y} and \ref{claim:sympl iso} that $\rho$ and $\quot{\psi_\rho}{y}$ are isomorphic. This proves Lemma \ref{le:A y}.
\end{proof}

\begin{proof}[Proof of Theorem \ref{thm:model}\reff{thm:model:inj}] Let $(H,\rho)$ and $(H',\rho')$ be two objects of $\SympRep{G}$ whose images under $\Model{G}$ are isomorphic. We choose an isomorphism $\Phi$ between these images. We define
\[y:=[e,0,0]\in Y_\rho,\quad[a',a'\phi',v']:=y':=\Phi(y).\]
By Remark \ref{rmk:stab} we have
\begin{eqnarray}\label{eq:Stab psi rho y H}&\Stab{\psi_\rho}{y}=H,&\\
\nn&\Stab{\psi_{\rho'}}{y'}\sub c_{a'}(H').&
\end{eqnarray}
Since $\Phi$ is $G$-equivariant, we have
\begin{equation}\label{eq:Stab psi rho y}\Stab{\psi_\rho}{y}=\Stab{\psi_{\rho'}}{y'}.\end{equation}
It follows that $H\sub c_{a'}(H')$. By considering $\Phi^{-1}$, an analogous argument shows that $H'$ is conjugate to a subgroup of $H$. Since $G$ is compact and $H$ and $H'$ are closed, these subgroups are compact. Therefore, applying Lemma \ref{le:N conj N'}, it follows that
\begin{equation}\label{eq:H c a' H'}H=c_{a'}(H').\end{equation}
Since $\mu_\rho(y)=\muL_\rho(e,0,0)=0$ and $\Stab{\psi_\rho}{y}=H$, the hypotheses of Lemma \ref{le:A y} are satisfied. Applying this lemma, it follows that $\rho$ is isomorphic to $\quot{\psi_\rho}{y}$. 

By Remark \ref{rmk:sympl quot Phi}\reff{rmk:sympl quot Phi:quot} $\quot{\psi_\rho}{y}$ is isomorphic to $\quot{\psi_{\rho'}}{y'}$.
\begin{claim}\label{claim:rho rho'} $\quot{\psi_{\rho'}}{y'}$ is isomorphic to $\rho'$.
\end{claim}
\begin{proof}[Proof of Claim \ref{claim:rho rho'}] By (\ref{eq:Stab psi rho y},\ref{eq:Stab psi rho y H},\ref{eq:H c a' H'}) we have $\Stab{\psi_{\rho'}}{y'}=c_{a'}(H')$. By Remark \ref{rmk:sympl quot Phi}\reff{rmk:sympl quot Phi:Ham} the map $\mu_{\rho'}\circ\Phi$ is a momentum map for $\psi_\rho$. Since $G$ is connected, the same holds for $Y_\rho$. It follows that $\mu_{\rho'}\circ\Phi-\mu_\rho$ is constantly equal to a central element of $\g^*$. At $y$ this map attains the value
\[\mu_{\rho'}(y')-\mu_\rho(y)=\mu_{\rho'}(y')-0,\]
which is thus a central element of $\g^*$. Hence the hypotheses of Lemma \ref{le:A y} are satisfied. Applying this lemma, the statement of Claim \ref{claim:rho rho'} follows.
\end{proof}
Combining this claim with what we already showed, it follows that $\rho$ is isomorphic to $\rho'$. 

Hence $\Model{G}$ induces an injective map between the sets of isomorphism classes. This proves Theorem \ref{thm:model}\reff{thm:model:inj}.
\end{proof}

\begin{proof}[Proof of Theorem \ref{thm:model}(\ref{thm:model:mor},\ref{thm:model:prop mor},\ref{thm:model:prop fct})] \reff{thm:model:mor} follows from a straightforward argument.\\

\reff{thm:model:prop mor}: Let $\rho$ and $\rho'$ be objects of $\SympRep{G}$, such that $\rho'$ is momentum proper and there exists a morphism $(g,T)$ from $\rho$ to $\rho'$. Let $Q\sub\h$ be compact. Equality \eqref{eq:Ad g * nu rho'} implies that
\begin{equation}\label{eq:nu rho -1 Q}\nu_\rho^{-1}(Q)=(\nu_{\rho'}\circ T)^{-1}\left(\Ad^*(g)(Q)\right).\end{equation}
The set $\Ad^*(g)(Q)$ is compact. By hypothesis $\nu_{\rho'}$ is proper, and by Claim \ref{claim:T Phi proper} the same holds for $T$. It follows that $\nu_{\rho'}\circ T$ is proper, and therefore, using \eqref{eq:nu rho -1 Q}, the set $\nu_\rho^{-1}(Q)$ is compact. Hence $\nu_\rho$ is proper, i.e., $\rho$ is momentum proper, as claimed.

Let now $(M,\om,\psi)$ and $(M',\om',\psi')$ be objects of $\HamEx{G}$, such that $\psi'$ is momentum proper and there exists a morphism $\Phi$ from $\psi$ to $\psi'$. We choose a momentum map $\mu'$ for $\psi'$. By definition, $\Phi$ is a proper $G$-equivariant symplectic embedding. It follows that $\mu'\circ\Phi$ is a proper momentum map for $\psi$. Hence $\psi$ is momentum proper. This proves \reff{thm:model:prop mor}.\\

\reff{thm:model:prop fct}: We prove ``$\then$'': Assume that $(H,\rho)$ is momentum proper, i.e., that $\nu_\rho$ is proper. Let $K\sub\g^*$ be compact. We denote by $i:H\to G$ the inclusion and by $i^*:\g^*\to\h^*$ the induced map. We define
\begin{equation}\label{eq:A}A:=\big\{(a,a\phi,v)\in T^*G\x V\,\big|\,\Ad^*(a)(\phi)\in K,\,i^*\phi=\nu_\rho(v)\big\}\sub(\muD_\rho)^{-1}(0).\end{equation}
We denote by $L_*:G\x T^*G\to T^*G$ the map induced by left translation. $A$ is a closed subset of
\[B:=L_*\big(G\x\Ad^*(G)(K)\big)\x\nu_\rho^{-1}\big(i^*\Ad^*(G)(K)\big).\]
Since $G$ is compact and $\Ad^*$ is continuous, the set $\Ad^*(G)(K)$ is compact. Since $i^*$ is continuous and $\nu_\rho$ is proper, it follows that $\nu_\rho^{-1}\big(i^*\Ad^*(G)(K)\big)$ is compact. Using that $L_*$ is continuous, it follows that $B$ is compact. It follows that $A$ is compact, and hence $\mu_\rho^{-1}(K)=\pi_\rho(A)$ is compact. Hence $\mu_\rho$ is proper. This proves ``$\then$''.\\

``$\follows$'': Assume that $\mu_\rho$ is proper. Let $Q\sub\h^*$ be compact. We choose a compact set $K\sub\g^*$ such that $i^*(K)=Q$. (We may e.g.~choose a linear complement $W\sub\g^*$ of $\ker i^*$ and define $K:=(i^*)^{-1}(Q)\cap W$.) Since $H$ is compact, the map $\pi_\rho:(\muD_\rho)^{-1}(0)\to Y_\rho$ is proper. It follows that $\mu_\rho\circ\pi_\rho$ is proper. Hence the set
\[\big(\mu_\rho\circ\pi_\rho\big)^{-1}(K)\]
is compact. This set agrees with $A$, defined as in \eqref{eq:A}. We denote by $\pr_2:T^*G\x V\to V$\label{pr2} the canonical projection. The set
\[C:=\big\{(e,\phi,v)\,\big|\,\phi\in K,\,i^*\phi=\nu_\rho(v)\big\}\]
is a closed subset of $A$, hence compact. It follows that $\pr_2(C)$ is compact. Since $i^*(K)=Q$, we have $\pr_2(C)=\nu_\rho^{-1}(Q)$. It follows that $\nu_\rho$ is proper. This proves ``$\follows$'', and completes the proof of \reff{thm:model:prop fct} and therefore of Theorem \ref{thm:model} (except for \reff{thm:model:surj}, which is proved in \cite[1.5.~Theorem]{KZ}).
\end{proof}

\section{Proof of Corollary \ref{cor:class crit} (classification of critical Hamiltonian actions)}\label{sec:proof:cor:class crit}

The well-definedness part of Corollary \ref{cor:class crit} follows from the next lemma.
\begin{lemma}\label{le:Tstar G}
\begin{enui}\item\label{le:Tstar G:obj} The functor $\Tstar{G}$ maps objects of $\Trans{G}$ to objects of $\HamCrit{G}$.
\item\label{le:Tstar G:mor} It maps morphisms of $\Trans{G}$ to morphisms of $\HamCrit{G}$.
\end{enui}
\end{lemma}
\begin{proof}[Proof of Lemma \ref{le:Tstar G}]\textbf{\reff{le:Tstar G:obj}:} Let $(Q,\theta)$ be an object of $\Trans{G}$. Then $\Tstar{G}(Q,\theta)$ is an object of $\HamEx{G}$. To see that is an object of $\HamExProp{G}$, we denote by $\g$ the Lie algebra of $G$. We choose a Finsler norm $\Vert\cdot\Vert$ on $T^*Q$ and a norm $|\cdot|$ on $\g^*$. We define $L^\theta_q$ as in \eqref{eq:L psi x} and 
\begin{equation}\label{eq:mu}\mu:T^*Q\to\g^*,\quad\mu(q,p):=pL^\theta_q.\end{equation}
This is a momentum map for the lifted $G$-action $\theta_*$. Since $\theta$ is transitive, the map $L^\theta_q$ is surjective. Therefore, an elementary argument using \eqref{eq:mu} and that $Q$ is compact, shows that
\[\sup\big\{\Vert p\Vert\,\big|\,(q,p)\in T^*Q:\,|\mu(q,p)|\leq C\big\}<\infty,\quad\forall C\in\R.\]
It follows that $\mu$ is proper. Therefore, $\Tstar{G}(Q,\theta)$ is an object of $\HamExProp{G}$. Since $Q$ is closed and $T^*Q$ deformation retracts onto $Q$, it follows that $\Tstar{G}(Q,\theta)$ is an object of $\HamCrit{G}$. This proves \reff{le:Tstar G:obj}.

\textbf{\reff{le:Tstar G:mor}:} Let $f:Q\to Q'$ be a morphisms of $\Trans{G}$, i.e., a $G$-equivariant diffeomorphism. The induced map $f_*:T^*Q\to T^*Q'$ is a $G$-equivariant symplectomorphism, hence a morphism of $\HamEx{G}$, and therefore of $\HamCrit{G}$. This proves \reff{le:Tstar G:mor} and therefore Lemma \ref{le:Tstar G}.
\end{proof}

By Lemma \ref{le:Tstar G} the restriction
\[\Tstar{G}:\Trans{G}\to\HamCrit{G}\]
is well-defined. The Chain Rule implies that it is functorial. In order to show that the map \eqref{eq:iso class } is a bijection, we need the following lemma. We define $\Sub{G}$\label{Sub G} to be the category whose objects are the closed subgroups of $G$ and whose morphisms between $H$ and $H'$ are those elements $g$ of $G$, satisfying $c_g(H)=H'$.\footnote{The composition of morphisms is given by the composition in $G$.} We define the functor\label{G/}
\[G/:\Sub{G}\to\Trans{G}\]
as follows. It maps an object $H$ to the quotient $G/H$, equipped with the canonical left $G$-action. Let $(H,H',g)$ be a morphism of $\Sub{G}$. We denote by $\pr_H:G\to G/H$\label{pr H} the canonical projection and define $G/(H,H',g):G/H\to G/H'$ to be the unique map satisfying
\begin{equation}\label{eq:G g}G/(H,H',g)\circ\pr_H=\pr_{H'}\circ R^{g^{-1}}.\end{equation}
This is a well-defined morphism of $\Trans{G}$. This construction is functorial. This defines the functor $G/$.

We define the functor\label{IIII G}
\[\IIII{G}:\Sub{G}\to\SympRep{G},\quad\IIII{G}(H):=\big(H,\{0\},0,0\big),\,\IIII{G}(g):=(g,0).\]
\begin{lemma}\label{le:nat iso} The target-restricted functor $\Model{G}\circ\IIII{G}:\Sub{G}\to\HamCrit{G}$ is well-defined and naturally isomorphic to $\Tstar{G}\circ G/:\Sub{G}\to\HamCrit{G}$.
\end{lemma}
\begin{proof}[Proof of Lemma \ref{le:nat iso}] By Corollary \ref{cor:class} the functor $\Model{G}\circ\IIII{G}$ takes values in $\HamExProp{G}$. Let $H$ be an object of $\Sub{G}$. The manifold part of $\Model{G}\circ\IIII{G}(H)$ is homotopy equivalent to the closed manifold $G/H$ and has dimension $2(\dim G-\dim H)$. Therefore, $\Model{G}\circ\IIII{G}(H)$ is critical. Hence $\Model{G}\circ\IIII{G}$ takes values in $\HamCrit{G}$, as claimed.

Let $H\in\Sub{G}$. We define $\muD_{H,\rho},Y_{H,\rho}$ as in (\ref{eq:muD H rho},\ref{eq:Y H rho}) and denote by $\pi_{H,\rho}:(\muD_{H,\rho})^{-1}(0)\to Y_{H,\rho}$ the canonical projection. We canonically identify $Y_{H,0}$ with the symplectic quotient of the Hamiltonian $H$-action on $T^*G$ induced by the right $H$-action on $G$. We define the map
\begin{equation}\label{eq:Phi H}\Phi_H:T^*(G/H)\to Y_{H,0},\quad\Phi_H(\BAR q,\BAR p):=\pi_{H,0}\big(q,\BAR pd\pr_H(q)\big),\end{equation}
where $q\in\BAR q$ is an arbitrary representative. This map is a symplectomorphism, see \cite[4.3.3 Theorem]{AM}. The map $\Phi_H$ is $G$-equivariant, and therefore an isomorphism of $\HamCrit{G}$. 
\begin{claim}\label{claim:Phi nat iso} The map $H\mapsto\Phi_H$ is a natural isomorphism between the functors $\Tstar{G}\circ G/$ and $\Model{G}\circ\IIII{G}$.
\end{claim}
\begin{pf}[Proof of Claim \ref{claim:Phi nat iso}] Let $(H,H',g)$ be a morphism of $\Sub{G}$. We show that
\begin{equation}\label{eq:Phi H'}\Phi_{H'}\circ\big((\Tstar{G}\circ G/)(H,H',g)\big)=\big(\Model{G}\circ\IIII{G}\big)(H,H',g)\circ\Phi_H.
\end{equation}
Let $(\BAR q,\BAR p)\in T^*(G/H)$. We choose a representative $q\in\BAR q$ and denote $q':=R^{g^{-1}}(q)$, $\BAR q':=\pr_{H'}(q')$, and $\BAR\phi:=G/(H,H',g)$. We have
\begin{align*}(\Tstar{G}\circ G/)(H,H',g)(\BAR q,\BAR p)&=\big(\BAR\phi(\BAR q),\BAR pd\BAR\phi(\BAR q)^{-1}\big),\\
\BAR\phi(\BAR q)&=\BAR q'\qquad\textrm{(using \eqref{eq:G g}),}
\end{align*}
and therefore,
\begin{align*}\Phi_{H'}\circ\big((\Tstar{G}\circ G/)(H,H',g)\big)(\BAR q,\BAR p)&=\pi_{H',0}\big(q',\BAR pd\BAR\phi(\BAR q)^{-1}d\pr_{H'}(q')\big)\qquad\textrm{(using \eqref{eq:Phi H})}\\
&=\pi_{H',0}\big(q',\BAR pd\pr_H(q)dR^{g^{-1}}(q)^{-1}\big)\,\textrm{(using \eqref{eq:G g}, Chain Rule)}\\
&=\Model{G}(g,0)\circ\pi_{H,0}\big(q,\BAR pd\pr_H(q)\big)\qquad\textrm{(using \eqref{eq:Model G})}\\
&=\big(\Model{G}\circ\IIII{G}\big)(H,H',g)\circ\Phi_H(\BAR q,\BAR p)\qquad\textrm{(using \eqref{eq:Phi H}).}\\
\end{align*}
Hence equality \eqref{eq:Phi H'} holds. This proves Claim \ref{claim:Phi nat iso} 
\end{pf} 
and therefore Lemma \ref{le:nat iso}.
\end{proof}

\begin{proof}[Proof of Corollary \ref{cor:class crit}]\label{proof:cor:class crit} We show that the functor $\Model{G}\circ\IIII{G}$ is essentially bijective. By Corollary \ref{cor:class} the inverse of the map \eqref{eq:iso class} is well-defined. The image of the set of isomorphism classes of $\HamCrit{G}$ under this inverse map is contained in the image of the map between isomorphism classes induced by $\IIII{G}$. This follows from the fact that the manifold part of $\Model{G}\big(H,V,\si,\rho\big)$ is homotopy equivalent to the closed manifold $G/H$ and has dimension $2(\dim G-\dim H)+\dim V$. It follows that $\Model{G}\circ\IIII{G}:\Sub{G}\to\HamCrit{G}$ is essentially surjective, i.e., it induces a surjective map between the sets of isomorphism classes.

Since $\IIII{G}$ is essentially injective, Corollary \ref{cor:class} implies that the functor $\Model{G}\circ\IIII{G}:\Sub{G}\to\HamCrit{G}$ is also essentially injective, and therefore essentially bijective, as claimed.

Using Lemma \ref{le:nat iso}, it follows that $\Tstar{G}\circ G/:\Sub{G}\to\HamCrit{G}$ is essentially bijective. Therefore, $\Tstar{G}:\Trans{G}\to\HamCrit{G}$ is essentially surjective. The functor $G/:\Sub{G}\to\Trans{G}$ is essentially surjective. This follows from the orbit-stabilizer theorem for $G$-actions on manifolds. Since $\Tstar{G}\circ G/:\Sub{G}\to\HamCrit{G}$ is essentially injective, it follows that $\Tstar{G}:\Trans{G}\to\HamCrit{G}$ is essentially injective, and therefore essentially bijective. This means that the map \eqref{eq:iso class } is bijective. This proves Corollary \ref{cor:class crit}.
\end{proof}

\section{Inverse of the classifying map}

The next result states that the inverse of the classifying map \eqref{eq:iso class} is induced by assigning to a Hamiltonian action its symplectic quotient representation at any suitable point. To state the result, let $G$ be a group, $X$ a set, $\psi$ a $G$-action on $X$, and $x\in X$. Recall that $\Stab{\psi}{x}$ denotes the stabilizer of $\psi$ at $x$. We call $x$ \emph{$\psi$-maximal} iff for every $y\in X$, $\Stab{\psi}{x}$ contains some conjugate of $\Stab{\psi}{y}$.

Let $G$ be a compact and connected Lie group, $(M,\om)$ a symplectic manifold, $\psi$ a symplectic $G$-action on $M$, and $x\in M$. Recall that $\quot{\psi}{x}$ denotes the symplectic quotient representation of $\psi$ at $x$, see \eqref{eq:quot}. The latter is a symplectic representation of $\Stab{\psi}{x}$. Hence the pair $\big(\Stab{\psi}{x},\quot{\psi}{x}\big)$ is an object of $\SympRep{G}$.

Assume now that $\psi$ is Hamiltonian. We call $x$ \emph{$\psi$-central} iff $\mu(x)$ is a central value of $\g^*$ for every momentum map $\mu$ for $\psi$. (If $M$ is connected then equivalently, \emph{there exists} such a $\mu$.)

\begin{prop}\label{prop:inv} Assume that $\psi$ is an object of $\HamExProp{G}$.
\begin{enui}
\item\label{prop:inv:ex} There exists a $\psi$-maximal and -central point.
\item\label{prop:inv:iso} Let $\psi$ and $\psi'$ be isomorphic objects of $\HamExProp{G}$, $x$ be $\psi$-maximal and -central point, and $x'$ be a $\psi'$-maximal and -central point. Then $\quot{\psi}{x}$ and $\quot{\psi'}{x'}$ are isomorphic.
\item\label{prop:inv:inv} The inverse map of \eqref{eq:iso class} is given by
\begin{eqnarray}\nn&\big\{\textrm{isomorphism class of }\HamExProp{G}\big\}\to\big\{\textrm{isomorphism class of }\SympRepProp{G}\big\},&\\
\label{eq:inv}&\Psi\mapsto[\quot{\psi}{x}],&
\end{eqnarray}
where $\psi$ is an arbitrary representative of $\Psi$, and $x$ is an arbitrary $\psi$-maximal and -central point.
\end{enui}
\end{prop}
\begin{Rmk} It follows from (\ref{prop:inv:ex},\ref{prop:inv:iso}) that the map \eqref{eq:inv} is well-defined.
\end{Rmk}
In the proof of Proposition \ref{prop:inv} we will use the following.
\begin{rmk}\label{rmk:rho quot} Let $\rho$ be an object of $\SympRep{G}$ and $y\in Y_\rho$ a $\psi_\rho$-maximal and -central point. Then $\rho$ and $\quot{\psi_\rho}{y}$ are isomorphic. To see this, we write $y=:[a,a\phi,v]$. By Remark \ref{rmk:stab} we have $\Stab{\psi_\rho}{y}\sub c_a(H)$. Since $y$ is $\psi_\rho$-maximal, $\Stab{\psi_\rho}{y}$ contains some conjugate of $\Stab{\psi_\rho}{[e,0,0]}=H$. Using Lemma \ref{le:N conj N'}, it follows that
\[\Stab{\psi_\rho}{y}=c_a(H).\]
Using that $y$ is $\psi_\rho$-central, the hypotheses of Lemma \ref{le:A y} are therefore satisfied. Applying this lemma, it follows that $\rho$ and $\quot{\psi_\rho}{y}$ are isomorphic, as claimed.
\end{rmk}
\begin{proof}[Proof of Proposition \ref{prop:inv}] \reff{prop:inv:ex}: Consider first the \textbf{case} in which there exists $\rho\in\SympRepProp{G}$, such that $\psi=\Model{G}(\rho)$. By Remark \ref{rmk:stab} the point $[e,0,0]$ is $\psi$-maximal. Since $\mu_\rho([e,0,0])=0$, this point is also $\psi$-central. This proves the statement in the special case.

The general situation can be reduced to this case, by using Theorem \ref{thm:model}\reff{thm:model:surj} (essential surjectivity), the fact that stabilizers are preserved under equivariant injections, and Remark \ref{rmk:sympl quot Phi}\reff{rmk:sympl quot Phi:Ham}. This proves \reff{prop:inv:ex}.\\

\reff{prop:inv:iso}: Consider first the \textbf{case} in which there exists an isomorphism from $\psi$ to $\psi'$ that maps $x$ to $x'$. Then it follows from Remark \ref{rmk:sympl quot Phi}\reff{rmk:sympl quot Phi:quot} that $\quot{\psi}{x}$ and $\quot{\psi'}{x'}$ are isomorphic.

In the general situation, using Theorem \ref{thm:model}\reff{thm:model:surj} and what we just proved, we may assume \Wlog that $\psi=\psi'=\psi_\rho=\Model{G}(H,\rho)$ for some object $\rho$ of $\SympRep{G}$. By Remark \ref{rmk:rho quot} we have $\quot{\psi_\rho}{x}\iso\rho\iso\quot{\psi_\rho}{x'}$. This proves \reff{prop:inv:iso}.\\

\reff{prop:inv:inv}: Remark \ref{rmk:rho quot} implies that \eqref{eq:inv} is a left-inverse for \eqref{eq:iso class}. Since \eqref{eq:iso class} is surjective, it follows that \eqref{eq:inv} is also a right-inverse. This proves \reff{prop:inv:inv} and completes the proof of Proposition \ref{prop:inv}.
\end{proof}

\section{Acknowledgments}

I would like to thank Yael Karshon for some useful comments and the anonymous referee for valuable suggestions.

\bibliographystyle{amsalpha}
\bibliography{amsj,references}

\section*{List of symbols}
The following list of symbols will hopefully help the reader navigate through this article. The last column shows the page on which the symbol is defined or occurs for the first time.\\

\begin{tabular}{llr}
symbol & meaning or r\^ole & p.\\
&&\\
$\Trans{G}$ & category of transitive $G$-actions on connected closed manifolds & \pageref{Trans G}\\
$\Ad,\Ad^*$ & adjoint, coadjoint $G$-representation & \pageref{Ad Ad*}\\
$c_g$ & conjugation by $g$ & \pageref{c g}\\
$\F{G}$ & functor from $\SYMPREP{G}$ to $\symprep{G}$ & \pageref{F G}\\
$G_x=\Stab{\psi}{x}$ & stabilizer of $x\in X$ under the $G$-action $\psi$ on $X$ & \pageref{G x}\\
$G/$ & functor from $\Sub{G}$ to $\Trans{G}$ & \pageref{G/}\\
$(g,T)$ & pair consisting of $g\in G$ and a linear symplectic map $T:V\to V'$ & \pageref{g T}\\
$\g,\h$ & Lie algebras of $G,H$ & \pageref{g h}\\
$H$ & closed subgroup of $G$ & \pageref{H}\\
$\HamEx{G}$ & category of exact Hamiltonian $G$-actions & \pageref{Ham Ex G}\\
$\HamExProp{G}$ & full subcategory of $\HamEx{G}$ consisting of momentum proper objects & \pageref{HamExProp G}\\
$\HamCrit{G}$ & full subcategory of $\HamExProp{G}$ consisting of critical objects & \pageref{HamCrit G}\\
$\HamContr{G}$ & full subcategory of $\HamEx{G}$ consisting of those objects $(M,\om,\psi)$ & \\
& for which $M$ is contractible & \pageref{HamContr G}\\
$\HamContrProp{G}$ & full subcategory of $\HamContr{G}$ consisting of momentum proper objects & \pageref{HamContrProp G}\\
$\II{G}{\rho}$ & canonical isomorphism from $\I{G}(\rho)=\rho$ to $\Model{G}(G,\rho)$ in $\HamEx{G}$ 
& \pageref{eq:A rho V Y rho}\\
$\III{G}$ & inclusion functor from $\SYMPREP{G}$ to $\SympRep{G}$ & \pageref{III G}\\
$\I{G}$ & inclusion functor from $\symprep{G}$ to $\HamContr{G}$ & \pageref{I G}\\
$\IProp{G}$ & inclusion functor from $\symprepprop{G}$ to $\HamContrProp{G}$ & \pageref{I G}\\
$\iota_{a,\phi}$ & inclusion from $V$ to $T^*G\x V$ induced by $(a,\phi)$ & \pageref{iota a phi}\\
$\IIII{G}$ & inclusion functor from $\Sub{G}$ to $\SympRep{G}$ & \pageref{IIII G}\\
$L_x=L^\psi_x$ & infinitesimal action at $x$ induced by $\psi$ & \pageref{eq:L psi x}\\
$\Model{G}$ & Hamiltonian $G$-model functor & \pageref{Model G}\\
$\big(M,\om,\psi\big)$ & Hamiltonian $G$-action & \pageref{M om psi}\\
$\muL$ & momentum map for the lifted left-translation action of $G$ on $T^*G$ & \pageref{mu L}\\
$\mu_\rho$ & momentum map for $\psi_\rho$ & \pageref{mu rho}\\
$\muD_{H,\rho}=\muD_\rho$ & momentum map for $\psiD_\rho$ & \pageref{mu D rho}\\
$\nu_\rho$ & unique momentum map for $\rho$ that vanishes at $0$ & \pageref{eq:nu rho}\\
$\om_Q$ & canonical symplectic form on $T^*Q$ & \pageref{om Q}\\
$\BAR\om_x$ & linear symplectic form on $V^\psi_x$ induced by $\om_x$ & \pageref{BAR om x}\\
$\pr_1$ & canonical projection from $T^*G\x V$ to $T^*G$ & \pageref{pr1}\\
$\pr_2$ & canonical projection from $T^*G\x V$ to $V$ & \pageref{pr2}\\
$\pr_H$ & canonical projection from $G$ to $G/H$ & \pageref{pr H}\\
$\pi_\rho$ & canonical projection from $(\muD_\rho)^{-1}(0)$ to $Y_\rho$ & \pageref{eq:pi rho}\\
$\psiD_\rho$ & diagonal $H$-action on $T^*G\x V$ induced by right translation on $G$ and $\rho$ & \pageref{psi D rho}\\
$R^g$ & right translation by $g\in G$ & \pageref{R g}\\
$\isotr{\psi}{x}$ & isotropy representation of $\psi$ at $x$ & \pageref{isotr psi x}\\
$\quot{\psi}{x}$ & symplectic quotient representation of $\psi$ at $x$ & \pageref{quot psi x}
\end{tabular}

\begin{tabular}{llr}
symbol & meaning or r\^ole & p.\\
&&\\
$\Stab{\psi}{x}=G_x$ & stabilizer of $x\in X$ under the $G$-action $\psi$ on $X$ & \pageref{G x}\\
$\Sub{G}$ & category of closed subgroups of $G$ & \pageref{Sub G}\\
$\SympRep{G}$ & category of symplectic representations of closed subgroups of $G$ & \pageref{SympRep}\\
$\SympRepProp{G}$ & full subcategory of $\SympRep{G}$ consisting of momentum proper objects & \pageref{SympRepProp G}\\
$\symprep{G}$ & category of symplectic $G$-representations & \pageref{symprep G}\\
$\symprepprop{G}$ & full subcategory of $\symprep{G}$ consisting of momentum proper objects & \pageref{symprepprop G}\\
$\SYMPREP{G}$ & category with objects symplectic $G$-representations and morphisms $(g,T)$ & \pageref{SYMPREP G}\\
$\Tstar{G}$ & $G$-cotangent functor & \pageref{T*G}\\
$(V,\si,\rho)$ & symplectic $H$-representation & \pageref{rho}\\
$V^\psi_x$ & $(\im L_x)^{\om_x}/\big(\im L_x\cap(\im L_x)^{\om_x}\big)$ & \pageref{eq:V psi x}\\
$W^\si$ & $\si$-symplectic complement of $W$ & \pageref{W si}\\
$\big(Y_\rho,\om_\rho,\psi_\rho)$ & Hamiltonian $G$-model associated to $\rho$ & \pageref{Y rho om rho psi rho}
\end{tabular}

\end{document}